\documentclass[12]{amsart}
\usepackage{amsmath,amssymb,amsthm,color,enumerate,comment,centernot,enumitem,url,cite}
\usepackage{graphicx,relsize}
\usepackage{mathtools}
\usepackage{array}

\newcommand{\ord}{\operatorname{ord}}

\newtheorem{thm}{Theorem}[section]

\newtheorem{cor}[thm]{Corollary}

\theoremstyle{definition}

\theoremstyle{remark}

\theoremstyle{definition}
    \newtheorem{defn}[thm]{Definition}

\newtheorem{rem}[thm]{Remark}

\theoremstyle{THM}

\newcommand{\abs}[1]{\left|{#1}\right|}

\def\R {{\mathbb R}}
\def\Z {{\mathbb Z}}

\def\Q {{\mathbb Q}}

\def\A {{\mathcal A}}

\def\F {{\mathbb F}}
\def\R {{\mathbb R}}

\def\Z {{\mathbb Z}}
\def\Q {{\mathbb Q}}

\def\red#1 {\textcolor{red}{#1 }}
\def\blue#1 {\textcolor{blue}{#1 }}

\numberwithin{equation}{section}

\def\Z {{\mathbb Z}}
\begin{document}

\title[Characteristic Polynomials of Simple Ordinary Abelian Varieties]{Characteristic Polynomials of Simple Ordinary Abelian Varieties over Finite Fields}

\author{Lenny Jones}
\address{Professor Emeritus, Department of Mathematics, Shippensburg University, Shippensburg, Pennsylvania 17257, USA}
\email[Lenny~Jones]{lkjone@ship.edu}

\date{\today}

\begin{abstract}
We provide an easy method for the construction of characteristic polynomials of simple ordinary abelian varieties $\A$ of dimension $g$ over a finite field $\F_q$, when $q\ge 4$ and $2g=\rho^{b-1}(\rho-1)$ for some prime $\rho\ge 5$, with $b\ge 1$. Moreover, we show that $\A$ is absolutely simple if $b=1$ and $\rho=5$, but $\A$ is not absolutely simple for arbitrary $\rho$ with $b>1$.
\end{abstract}


\maketitle
\section{Introduction}\label{Section:Intro}
For positive integers $g$ and $q$, we say $f(t)\in \Z[t]$ is a \emph{$q$-polynomial} if
\begin{align}\label{f}\begin{split}
f(t)&=t^{2g}+a_{1}t^{2g-1}+\cdots +a_gt^g+a_{g-1}qt^{g-1}+\cdots +a_1q^{g-1}t+q^g\\
&=t^{2g}+a_gt^g+q^g+\sum_{j=1}^{g-1}a_j\left(t^{2g-j}+q^{g-j}t^j\right),
\end{split}
\end{align}
and all zeros of $f(t)$ have modulus $q^{1/2}$. Not all polynomials of the form \eqref{f} are $q$-polynomials since the condition on the moduli of the zeros of $f(t)$ imposes severe restrictions on its coefficients. For example,
\[f(t)=t^6+t^5+t^4+5t^3+2t^2+4t+8\]
 has the form \eqref{f} with $g=3$ and $q=2$, and although $f(t)$ has four zeros with modulus $2^{1/2}$, $f(t)$ has two real zeros, neither of which has modulus $2^{1/2}$.

Most likely, D. H. Lehmer \cite{Lehmer} in 1932 was the first mathematician to investigate $q$-polynomials. He was mainly interested in $q$-polynomials with the property that all zeros have the form  $q^{1/2}\zeta$, for some root of unity $\zeta$. Lehmer called such polynomials \emph{quasi-cyclotomic}. Since then, certain  $q$-polynomials have become central to the study of abelian varieties over finite fields.

Throughout this article, we let $k$ denote the finite field $\F_q$, where $q=p^n$ for some prime $p$ and positive integer $n$.  It is well known from the Honda-Tate theorem \cite{Honda, Tate, Waterhouse,WM} that the isogeny class of an abelian variety $\A$ of dimension $g$ over $k$ is determined by the characteristic polynomial 
$f_{\A}(t)\in \Z[t]$ of its Frobenius endomorphism \cite{Tate,WM}. With a slight abuse of terminology, we refer here to $f_{\A}(t)$ as the \emph{characteristic polynomial of $\A$}. It follows from the Weil conjectures \cite{Weil,Hart} (conjectured in 1949 by Weil, and subsequently proven by Dwork \cite{Dwork}, Grothendieck \cite{Grothendieck}, Deligne \cite{Deligne} and others) that $f_{\A}(t)$ has the form in \eqref{f} \cite{SMZ}, and all zeros of $f_{\A}(t)$ have modulus $q^{1/2}$.
In other words, $f_{\A}(t)$ is a $q$-polynomial, and if a $q$-polynomial $f(t)$ is such that $f(t)=f_{\A}(t)$, for some abelian variety $\A$ over $k$, then $f(t)$ is called a \emph{Weil polynomial}. However, not every $q$-polynomial is a Weil polynomial, since additional restrictions on the coefficients of $f_{\A}(t)$ are imposed by the Honda-Tate theorem. For example, it is straightforward to verify that
\[f(t)=t^4+2t^3+2t^2+16t+64\] is an irreducible $q$-polynomial with $g=2$ and $q=8$, but $f(t)$ is not the characteristic polynomial of an abelian variety over $k=\F_8$ \cite{MN,Ruck}, and so $f(t)$ is not a Weil polynomial.
\begin{rem}
  We caution the reader that while we have chosen to follow \cite{HZ} in making no distinction between Weil polynomials and characteristic polynomials $f_{\A}(t)$, certain authors \cite{H,HS,MN} have given a broader definition of Weil polynomials.
\end{rem}

For small dimensions, explicit necessary and sufficient conditions on the coefficients of \eqref{f} have been given \cite{WM,Ruck,MN,H,HS,SMZ} to determine which irreducible $q$-polynomials actually arise as characteristic polynomials of abelian varieties. Typically, Newton polygons are useful in the derivation of such conditions. For larger dimensions, however, this task becomes increasingly more difficult, and a complete characterization in arbitrary dimension seems infeasible.  

An abelian variety $\A$ over $k$ of dimension $g$ is called \emph{simple} if $\A$ has no proper nontrivial subvarieties over $k$, and $\A$ is called \emph{absolutely simple} if $\A$ is simple over the algebraic closure of $k$. Additionally, $\A$ is called \emph{ordinary} if the rank of its group of $p$-torsion points over the algebraic closure of $k$ equals $g$.

 It is the purpose of this article to present an easy method for the construction of characteristic polynomials $f_{\A}(t)$, where $\A$ is a simple ordinary abelian variety of dimension $g$ over $k$, such that $q\ge 4$ and $2g=\rho^{b-1}(\rho-1)$ for some prime $\rho\ge 5$, with $b\ge 1$. More precisely, we prove
 \begin{thm}\label{Thm:Main1}
 Let $\rho\ge 5$ be a prime, let $b\ge 1$ be an integer and let $2g=\rho^{b-1}(\rho-1)$.
     Let $r$ be a prime such that $r$ is a primitive root modulo $\rho^2$. Let $p$ be a prime, and let $n$ be a positive integer such that $q:=p^n\ge 4$ and $q\equiv 1 \pmod{r}$.
    Let $m$ be an integer such that
    $m\not \equiv -1/r\pmod{p}$ and
     \[0\le m\le \dfrac{2q^{\rho^{b-1}/2}\left(q^{\rho^{b-1}/2}-1\right)-1}{r}.\]
    Define
    \begin{equation}\label{Thm:f}
    f(t):=t^{2g}+\left(mr+1\right)t^g+q^g+\sum_{j=1}^{g-1}a_j\left(t^{2g-j}+q^{g-j}t^j\right),
    \end{equation}
    where
    \begin{equation}\label{Eq:ajconditions}
    a_j=\left\{\begin{array}{cl}
      1 & \mbox{if $j\equiv 0 \pmod{\rho^{b-1}}$}\\
      0 & \mbox{otherwise}
    \end{array}\right.\quad \mbox{for $j\in \{1,2,\ldots, g-1\}$.}
    \end{equation}
    Then $f(t)$  is the characteristic polynomial $f_{\A}(t)$ of a simple ordinary abelian variety $\A$ of dimension $g$ over the field $k=\F_q$.
        Furthermore,
        \begin{enumerate}
        \item \label{IMain:1} if $b=1$ and $\rho=5$ (so that $g=2$), then $\A$ is absolutely simple,
        \item \label{IMain:2} if $b>1$ and $\rho$ is arbitrary, then $\A$ is not absolutely simple.
        \end{enumerate}
  \end{thm}


\section{Preliminaries} \label{Section:Prelims}

For any integer $N\ge 1$, let $\Phi_N(x)$ denote the cyclotomic polynomial of index $N$.
\begin{thm}\label{Thm:Guerrier} {\rm \cite{Guerrier}}
  Let $r$ be a prime such that $r\nmid n$. Let $\ord_n(r)$ denote the order of $r$ modulo $n$. Then $\Phi_n(x)$ factors modulo $r$ into a product of $\phi(n)/\ord_n(r)$ distinct irreducible polynomials, each of degree $\ord_n(r)$. 
\end{thm}

\begin{cor}\label{Cor:Guerrier}
  Let $\rho\ge 3$ and $r$ be primes such that $r$ is a primitive root modulo $\rho^2$. Let $b\ge 1$ be an integer. If $f(x)\in \Z[x]$ is monic with $f(x)\equiv \Phi_{\rho^b}(x) \pmod{r}$, then $f(x)$ is irreducible over $\Q$.
  \end{cor}
  \begin{proof}
    Since $r$ is a primitive root modulo $\rho^2$, $r$ is a primitive root modulo $\rho^e$ for all $e\ge 1$ \cite{Burton}. That is, $\ord_{\rho^e}(r)=\phi\left(\rho^e\right)$. Thus, it follows from Theorem \ref{Thm:Guerrier} that $f(x)$ is irreducible modulo $r$, and hence, irreducible over $\Q$.
  \end{proof}
  \begin{defn}
    We say that $f(x)\in \R[x]$ is \emph{reciprocal} if $f(x)=x^{\deg{f}}f(1/x)$.
  \end{defn}
\begin{thm}{\rm \cite{LL}}\label{Thm:LL}
 Let $N\ge 2$ be an integer, and let
 \[P_N(x)=\sum_{j=0}^N c_jx^j\in \R[x]\]
 be reciprocal with $c_N\ne 0$. If there exists $\delta\in \R$ with $c_N\delta\ge 0$ and $\abs{c_N}\ge \abs{\delta}$, such that
 \[\abs{c_N+\delta}\ge \sum_{j=1}^{N-1}\abs{c_j+\delta-c_N},\]
 then all zeros of $P_N(x)$ are on the unit circle.
\end{thm}

 \begin{thm}{\rm \cite{DH}}\label{Thm:DH}
Let $n$ and $g$ be positive integers.
Let $p$ be a prime, and let $q=p^n$.
  Suppose that $f(t)\in \Z[t]$ is monic with $\deg(f)=2g$, and that $a_g$ is the coefficient of $t^g$. If all zeros of $f(t)$ have modulus $q^{1/2}$ and $\gcd\left(a_g,p\right)=1$, then $f(t)$ is the characteristic polynomial $f_{\A}(t)$ of an ordinary abelian variety $\A$ of dimension $g$ over $k$. 
\end{thm}

By the Honda-Tate theorem, we have the following:
 \begin{thm}{\rm \cite{Howe,HZ}}\label{Thm:HZ}
 Let $\A$ be an ordinary abelian variety of dimension $g$ over $k$, and let $f_{\A}(t)$ be the characteristic polynomial of $\A$.
 Then $\A$ is simple if and only if $f_{\A}(t)$ is irreducible.
  \end{thm}

  The following theorem gives an easy test for determining whether a simple ordinary abelian variety $\A$ of dimension 2 over $k$ is absolutely simple.
  \begin{thm}{\rm \cite{HZ,MN}}\label{Thm:HZ1}
   Let $\A$ be a simple ordinary abelian variety of dimension 2 over $k$ with characteristic polynomial $f_{\A}(t)=t^4+a_1t^3+a_2t^2+a_1qt+q^2$. Then $\A$ is absolutely simple if and only if $a_1^2\not \in \{0,q+a_2,2a_2,3a_2-3q\}$.
   \end{thm}
The next theorem addresses when a simple ordinary abelian variety $\A$ of arbitrary dimension over $k$ is absolutely simple.
  \begin{thm}{\rm \cite{HZ}}\label{Thm:HZ2}
    Let $\A$ be a simple ordinary abelian variety over $k$ with characteristic polynomial $f_{\A}(t)$. Suppose that $f_{\A}(\theta)=0$. Then
    $\A$ is absolutely simple if and only if $\Q(\theta)=\Q\left(\theta^d\right)$ for all integers $d>0$.
  \end{thm}

  \section{The Proof of Theorem \ref{Thm:Main1}}
  \begin{proof}
   We first prove that $f(t)$ is a $q$-polynomial. To accomplish this task, it is enough to show that all zeros of $f(t)$ have modulus $q^{1/2}$, since it is obvious that $f(t)$ has the form \eqref{f}.
    Let $a_g:=mr+1$. Since
    \[\left \lfloor \dfrac{g-1}{\rho^{b-1}}\right \rfloor=\dfrac{g}{\rho^{b-1}}-1=\dfrac{\rho-3}{2},\]
    we have from \eqref{Eq:ajconditions} that
\[f(t)=t^{2g}+a_gt^g+q^g+\sum_{u=1}^{(\rho-3)/2}\left(t^{2g-u\rho^{b-1}}+q^{g-u\rho^{b-1}}t^{u\rho^{b-1}}\right).\]
Then
 \begin{align*}
    F(t):&=f\left(q^{1/2}t\right)\\
    &=q^gt^{2g}+q^{g/2}a_gt^g+q^g+\sum_{u=1}^{(\rho-3)/2}q^{(2g-u\rho^{b-1})/2}\left(t^{2g-u\rho^{b-1}}+t^{u\rho^{b-1}}\right)
    \end{align*}
    is reciprocal. Let
    \[S=\abs{c_N+\delta}-\sum_{j=1}^{N-1}\abs{c_j+\delta-c_N},\]
    where $N=2g$, $c_N=\delta=q^g$ and $c_j$ is the coefficient of $t^j$ in $F(t)$, for $j=1,2,\ldots, N-1$. Then, using the fact that
    \[a_g\le 2q^{\rho^{b-1}/2}\left(q^{\rho^{b-1}/2}-1\right),\] we have
   \begin{align*}
    S&=2q^g-2\left(q^{\left(2g-\rho^{b-1}\right)/2}+q^{\left(2g-2\rho^{b-1}\right)/2}+\cdots +q^{\left(2g-\left(\frac{\rho-3}{2}\right)\rho^{b-1}\right)/2}\right)-a_gq^{g/2}\\
    &=2q^g-2q^{\left(2g-\left(\frac{\rho-3}{2}\right)\rho^{b-1}\right)/2}\left(\left(q^{\rho^{b-1}/2}\right)^{(\rho-5)/2}+ \cdots + \left(q^{\rho^{b-1}/2}\right)+1\right)-a_gq^{g/2}\\
    &=2q^g-2q^{\left(2g-\left(\frac{\rho-3}{2}\right)\rho^{b-1}\right)/2}\left(\dfrac{\left(q^{\rho^{b-1}/2}\right)^{\frac{\rho-3}{2}}-1}
    {q^{\rho^{b-1}/2}-1}\right)-a_gq^{g/2}\\
    &\ge 2q^g-2q^{\left(2g-\left(\frac{\rho-3}{2}\right)\rho^{b-1}\right)/2}\left(\dfrac{\left(q^{\rho^{b-1}/2}\right)^{\frac{\rho-3}{2}}-1}
    {q^{\rho^{b-1}/2}-1}\right)-2q^{\rho^{b-1}/2}\left(q^{\rho^{b-1}/2}-1\right)q^{g/2}\\
    &=\dfrac{2q^{(2g+\rho^{b-1})/2}-4q^g-2q^{(g+3\rho^{b-1})/2}+4q^{(g+2\rho^{b-1})/2}}{q^{\rho^{b-1}/2}-1}\\
    &=\dfrac{2q^{(g+2\rho^{b-1})/2}\left(q^{(g-2\rho^{b-1})/2}-1\right)\left(q^{\rho^{b-1}/2}-2\right)}{q^{\rho^{b-1}/2}-1}\\
    &\ge 0,
    \end{align*}
    since $g\ge 2\rho^{b-1}$ and $q\ge 4$. Hence, from Theorem \ref{Thm:LL}, all zeros of $F(t)$ are on the unit circle, and consequently, all zeros of $f(t)$ have modulus $q^{1/2}$. 

        We now show that $f(t)$ is a Weil polynomial. In particular, we prove that $f(t)=f_{\A}(t)$ for a simple ordinary abelian variety of dimension $g$ over $k$. Observe that $\gcd\left(a_g,p\right)=1$ since $m\not \equiv -1/r\pmod{p}$, and so we deduce from Theorem \ref{Thm:DH} that $f(t)=f_{\A}(t)$, where $\A$ is an ordinary abelian variety of dimension $g$ over $k$.
        Since $r$ is a primitive root modulo $\rho^2$ and $f_{\A}(t)\equiv \Phi_{\rho^b}(t) \pmod{r}$, it follows from Corollary \ref{Cor:Guerrier}  that $f_{\A}(t)$ is irreducible over $\Q$. Therefore, since $\A$ is ordinary, we conclude that $\A$ is simple by 
        Theorem \ref{Thm:HZ}.

        If $b=1$ and $g=2$, we have from \eqref{Thm:f} that
        \[f_{\A}(t)=t^4+t^3+(mr+1)t^2+qt+q^2,\] where $a_1=1$ and $a_2=mr+1$. Thus, it is easy to see from Theorem \ref{Thm:HZ1} that $\A$ is absolutely simple, which proves \emph{(\ref{IMain:1})}.

         Finally, to establish \emph{(\ref{IMain:2})}, suppose that $f_{\A}(\theta)=0$. Since $b>1$, it follows from \eqref{Thm:f} and the irreducibility of $f_{\A}(t)$ that the minimal polynomial of $\theta^{\rho^{b-1}}$ has degree $\rho-1$. Hence, $\Q\left(\theta^{\rho^{b-1}}\right)\ne \Q(\theta)$, and $\A$ is not absolutely simple by Theorem \ref{Thm:HZ2}.
  \end{proof}






\end{document}